\documentclass[letterpaper, 10 pt, conference]{ieeeconf}  % Comment this line out if you need a4paper

\IEEEoverridecommandlockouts                              % This command is only needed if
\usepackage{amstext,amssymb,amsmath,amsbsy}
\usepackage{hyperref}
\usepackage{graphics} % for pdf, bitmapped graphics files
\usepackage{epsfig} % for postscript graphics files
\usepackage{mathptmx} % assumes new font selection scheme installed
\usepackage{times} % assumes new font selection scheme installed
\usepackage{amsmath} % assumes amsmath package installed
\usepackage{amssymb}  % assumes amsmath package installed

\usepackage{cleverref}
\usepackage{cases}
\usepackage{color}

\usepackage{tikz}

\usepackage{amsmath,amsfonts,amssymb,epsfig, textcomp}
\usepackage{tikz}
\usepackage{amscd}
\usepackage{amsfonts}

\usepackage{amsmath}

\usepackage{indentfirst}
\usepackage{verbatim}

\usepackage{enumerate}

\usepackage[OT1]{fontenc}
\usepackage[latin1]{inputenc}
\usepackage[english]{babel}
\usepackage{amssymb}

\newtheorem{theorem}{Theorem}[section]
\newtheorem{lemma}{Lemma}[section]

\newtheorem{proposition}{Proposition}[section]

\newtheorem{remark}{Remark}[section]

\newcommand{\cA}{{\mathcal A}}

\newcommand{\mN}{\mathbb{N}}

      \newcommand{\eps}{\varepsilon}
      \newcommand{\mR}{\mathbb{R}}

 \newcommand{\cB}{{\mathcal B}}

      \newcommand{\mC}{\mathbb{C}}

      \newcommand{\dsp}{\displaystyle}

      \makeatletter
      \def\@setcopyright{}
      \def\serieslogo@{}
      \makeatother
   \newcommand{\tr}{^\mathsf{T}}

\newcommand{\cK}{{\mathcal K}}
\newcommand{\D}{{\mathcal T}}

\newcommand{\B}{\mathrm{B}}

\title{\LARGE \bf
Null-controllability, exact controllability, and stabilization  of hyperbolic systems for the optimal time
}

\author{Jean-Michel Coron$^{1}$ and Hoai-Minh Nguyen$^{2}$% <-this % stops a space
\thanks{$^{1}$ Jean-Michel Coron is with Sorbonne Universit\'{e}s, UPMC Univ Paris 06, UMR 7598, Laboratoire
Jacques-Louis Lions, 4 place Jussieu, F-75252, Paris, France
{\tt \small coron@ann.jussieu.fr.}}
\thanks{$^{2}$Hoai-Minh Nguyen is  with the Mathematics Department, Ecole Polytechnique F\'ed\'erale de Lausanne, EPFL,  SB MATHAA CAMA, Station 8,  CH-1015 Lausanne, Switzerland
        {\tt\small hoai-minh.nguyen@epfl.ch}}%
}

\begin{document}

\maketitle
\thispagestyle{empty}
\pagestyle{empty}

%%%%%%%%%%%%%%%%%%%%%%%%%%%%%%%%%%%%%%%%%%%%%%%%%%%%%%%%%%%%%%%%%%%%%%%%%%%%%%%%
\begin{abstract}

In this paper, we discuss our recent works on  the null-controllability, the exact controllability, and the stabilization of linear hyperbolic systems in one dimensional space using boundary controls on one side for the optimal time. Under precise and generic assumptions on the boundary conditions on the other side, we first obtain the optimal time for the null and the exact controllability for these systems for a generic source term. We then  prove the null-controllability and the exact controllability for any time greater than the optimal time and for any source term.  Finally, for homogeneous systems, we design feedbacks which stabilize the systems and bring them to the zero state at the optimal time. Extensions for the non-linear homogeneous system are also discussed.

\end{abstract}

\section{Introduction}

Linear hyperbolic systems in one dimensional space are frequently used
in modeling of many systems such as traffic flow, heat exchangers,  and fluids in open channels. The
stability and boundary stabilization of these hyperbolic systems
have been studied intensively in the literature, see,  e.g.,   \cite{BC16} and the references therein. In this paper, we are interested in  the null-controllability, the exact controllability, and the stabilization at finite time of linear hyperbolic systems in one dimensional space using boundary controls on one side.  More precisely, we consider the system, for
$ (t, x)  \in \mR_+ \times (0, 1)$,
\begin{equation}\label{Sys-1}
\partial_t w (t, x) =  \Sigma(x) \partial_x w (t, x) + C(x) w(t, x).
\end{equation}
Here $w = (w_1, \cdots, w_n)\tr : \mR_+ \times (0, 1) \to \mR^n$ ($n \ge 2$),   $\Sigma$ and $C$ are   $(n \times n)$ real matrix-valued functions defined in $[0,1]$. We assume that for every $x \in [0, 1]$,  $\Sigma(x)$ is diagonal with $m \ge 1$  distinct positive eigenvalues,  and $k = n - m \ge 1$  distinct negative eigenvalues. Using Riemann coordinates, one might assume that $\Sigma(x)$ is of the form
\begin{equation}\label{form-A}
\Sigma(x) = \mbox{diag} \big(- \lambda_1(x), \cdots, - \lambda_{k}(x),  \lambda_{k+1}(x), \cdots,  \lambda_{n}(x) \big),
\end{equation}
where
\begin{equation}\label{relation-lambda}
-\lambda_{1}(x) < \cdots <  - \lambda_{k} (x)< 0 < \lambda_{k+1}(x) < \cdots < \lambda_{k+m}(x).
\end{equation}
Throughout the paper, we assume that
\begin{equation}\label{cond-lambda}
\mbox{$\lambda_i$  is Lipschitz on $[0, 1]$  for $1 \le i \le n \,  (= k + m)$.}
\end{equation}
We also assume that
\begin{equation}
C \in \big( L^\infty([0, 1]) \big)^{n \times n}.
\end{equation}
We are interested in the following type of boundary conditions and boundary controls. The boundary conditions at $x = 0$ are given by, for $t \ge 0$,
\begin{equation}\label{bdry-0}
(w_1, \cdots, w_k)\tr (t, 0)  = B(w_{k+1}, \cdots, w_{k+m})\tr (t, 0)
\end{equation}
for some $(k \times m)$ real {\it constant} matrix $B$, and the boundary controls  at $x = 1$ are, for $t \ge 0$,
\begin{equation}\label{bdry-1}
w_{k+1}(t, 1) = W_{k+1}(t),  \quad \dots, \quad w_{k + m}(t, 1) = W_{k+m}(t),
\end{equation}
where $W_{k +1}, \dots, W_{k + m}$ are controls.

Let us recall that the control system \eqref{Sys-1}, \eqref{bdry-0}, and \eqref{bdry-1} is null-controllable (resp. exactly controllable) at the time $T>0$ if, for every initial data $w_0: (0,1)\to \mathbb{R}^n$ in $[L^2(0,1)]^n$ (resp. for every initial data $w_0: (0,1 )\to \mathbb{R}^n$ in $[L^2(0,1)]^n$  and for every (final) state $w_T:  (0,1 )\to \mathbb{R}^n$  in $[L^2(0,1)]^n$), there is a control $W=(W_{k+1},\ldots,W_{k+m})\tr :(0,T)\to \mathbb{R}^m$ in $[L^2(0,T)]^m$ such that the solution
of \eqref{Sys-1}, \eqref{bdry-0}, and \eqref{bdry-1} satisfying $w(t=0,x)=w_0(x)$ vanishes (resp. reaches $w_T$) at the time $T$: $w(t=T,x)=0$ (resp. $w(t = T, x)= w_T(x)$). Similar definitions hold for $w_0 \in \big[L^\infty(0, 1) \big]^n$ (resp. $w_0, w_T \in  \big[L^\infty(0, 1) \big]^n$) with $W \in \big[L^\infty(0, T) \big]^m$.
Set
\begin{equation}\label{def-tau}
\tau_i :=  \int_{0}^1 \frac{1}{\lambda_i(\xi)}  \, d \xi  \mbox{ for } 1 \le i \le n.
\end{equation}

The exact controllability, the null-controllability, and the boundary stabilization problem of  hyperbolic system in one dimension  have been widely investigated in the literature for almost half a century. The pioneer works date back to Rauch and Taylor \cite{RT74} and Russell \cite{Russell78}.   In particular, it was shown, see \cite[Theorem 3.2]{Russell78},  that the system \eqref{Sys-1}, \eqref{bdry-0}, and \eqref{bdry-1} is null-controllable if for the time
$$
T \ge \tau_k + \tau_{k+1},
$$
and is exact controllable if $k=m$ and $B$ is invertible. The extension of  this result for quasilinear system was initiated by
Greenberg and Li \cite{GL84} and Slemrod \cite{Slemrod83}.

Concerning the stabilisation of \eqref{Sys-1}, many articles are based on the boundary conditions with the following specific form
\begin{equation}\label{bdry-S}
\left(\begin{array}{c}
    w_- (t, 0) \\
    w_+ (t, 1)
  \end{array}\right)
= G
 \left( \begin{array}{c}
    w_+(t, 1) \\
    w_- (t, 0)
  \end{array} \right),
\end{equation}
where $G: \mR^n \to \mR^n$ is a suitable smooth vector field. Three approaches have been proposed to deal with \eqref{bdry-S}.  The first one is based on the characteristic method.  This method was previously investigated in Greenberg and Li \cite{GL84} for $2 \times 2$ systems and Qin \cite{Qin85} (see also Li \cite{Li94})  for a generalization to $n \times n$ homogeneous nonlinear hyperbolic systems in the framework of $C^1$-norm. The second approach is based on Lyapunov functions see
\cite{BCN07, CBN08} (see also \cite{LS02, BCN99}). The third approach \cite{CoronNg15} is based on the study of delay equations works for $W^{2,p}$-norm with $p \ge 1$. These works typically impose restrictions on the magnitude of the coupling coefficients.

This restriction was overcome via backstepping approach. This was first proposed by Coron et al. \cite{CVKB13} for $2\times 2$ system $(m=k=1)$.  Later this approach has been extended and now can be applied for general pairs $(m,k)$, see \cite{MVK13, HMVK16, AM16, CHO17}.  In \cite{CVKB13}, the null-controllability is achieved via a feedback law for the time $\tau_1+\tau_2$  with $m = k = 1$ via backstepping approach.   In \cite{HMVK16}, the authors considered the case where $\Sigma$ is constant and   obtained feedback laws for the null-controllability at the time
\begin{equation}\label{def-T1}
T_1 : = \tau_k + \sum_{l = 1}^m \tau_{k + l}.
\end{equation}
 It was  later showed in \cite{AM16, CHO17} that one can reach the null-controllability at the time
\begin{equation}\label{def-T2}
T_2 := \tau_k + \tau_{k+1}.
\end{equation}

Set
\begin{equation}\label{def-Top}
T_{opt} := \left\{ \begin{array}{r}  \dsp \max \big\{ \tau_1 + \tau_{m+1}, \dots, \tau_k + \tau_{m+k}, \\[6pt] \tau_{k+1} \big\}  \mbox{ if } m \ge k, \\[6pt]
\dsp \max \big\{ \tau_{k+1-m} + \tau_{k+1}, \tau_{k+2-m} + \tau_{k+2},  \\[6pt]  \dots,
\tau_{k} + \tau_{k+m} \big\}   \mbox{ if } m < k.
\end{array} \right.
\end{equation}
In this paper, we report our recent works \cite{CoronNg19, CoronNg19-2, CoronNg20,CoronNg20-L} on  the null-controllability, the exact controllability, and the stabilization of \eqref{Sys-1}, \eqref{bdry-0}, and \eqref{bdry-1}. We show that
the null-controllability holds at $T_{opt}$ for generic $B$ and $C$ and the null-controllability holds for any $T> T_{opt}$ under a precise condition on $B$ ($B \in {\mathcal B}$ given in \eqref{def-B}), which holds for almost every matrix $k \times m$ matrix  $B$.  Similar conclusions holds for the exact controllability (with $B \in {\mathcal B}_e$ given in \eqref{def-Be}) under  natural, additional condition $m \ge k$. When the system is homogeneous,  we show that the null-controllability is achieved via a time-independent feedback and there are Lyapunov's functions associated with these feedbacks.  This result also holds for quasilinear setting.  The starting point of our approach in the inhomogeneous case is the backstepping approach.

\begin{remark} \rm The backstepping approach for the control of partial differential equations was pioneered by  Miroslav Krstic and his coauthors (see \cite{Krstic08} for a concise introduction).  The backstepping method is now frequently used for various control problems modeling by partial differential equations in one dimension. For example, it has been also used to stabilize the  wave equation  \cite{KGBS08, SK09, SCK10}, the parabolic equations in \cite{SK04,SK05},  nonlinear parabolic equations  \cite{Vazquez08},  and to  obtain the null-controllability of the heat equation \cite{CoronNg17}. The standard backstepping approach relies on the Volterra transform of the second kind. It is worth noting that, in some situations, more general transformations have to be  considered as for Korteweg-de Vries equations  \cite{CoronC13}, Kuramoto--Sivashinsky equations \cite{Coron15}, Schr\"{o}dinger's equation \cite{CoronGM18}, and hyperbolic equations with internal controls \cite{2019-Zhang-preprint}.
\end{remark}

\section{STATEMENT OF THE MAIN RESULTS}

Define
\begin{equation}\label{def-B}
{\mathcal B}: = \big\{B \in \mR^{k \times m}; \mbox{ \eqref{cond-B-1} holds for  $1 \le i \le  \min\{k, m-1\}$} \big\},
\end{equation}
and
\begin{equation}\label{def-Be}
{\mathcal B}_e: = \big\{B \in \mR^{k \times m}; \mbox{ \eqref{cond-B-1} holds for $1 \le i \le  k$} \big\},
\end{equation}
where
\begin{multline}\label{cond-B-1}
\mbox{ the $i \times i$  matrix formed from the last $i$} \\\mbox{ columns and the last $i$ rows of $B$  is invertible.}
\end{multline}

We first show that the system  \eqref{Sys-1}, \eqref{bdry-0}, and \eqref{bdry-1} is null-controllable for $B \in {\mathcal B}$ and is exact controllable for $B \in {\mathcal B}_e$ at the time $T_{opt}$ generically. More precisely, we have \cite[Theorem 1.1]{CoronNg19}:

\begin{theorem}\label{thm1} Assume that \eqref{relation-lambda} and \eqref{cond-lambda} hold.
We have
\begin{itemize}
\item[i)] for each $B \in {\mathcal B}$, outside a discrete set of $\gamma$ in $\mR$, the control system \eqref{Sys-1} with $C$ replaced by $\gamma C$, \eqref{bdry-0}, and \eqref{bdry-1} is null-controllable at the time $T_{opt}$.

\item[ii)] for each $\gamma$ outside a discrete set in $\mR$, outside a set of zero measure of $B$ in ${\mathcal B}$, the control system \eqref{Sys-1} with $C$ replaced by $\gamma C$, \eqref{bdry-0}, and \eqref{bdry-1} is null-controllable at the time $T_{opt}$.
\end{itemize}
\end{theorem}

\begin{remark} \rm In the case $m = 1$, we can show that there exists a (linear) time independent feedback which yields the null-controllability at the time $T_{opt}$ for \eqref{Sys-1}, \eqref{bdry-0}, and \eqref{bdry-1} for all $B$, see \cite[Theorem 1.1]{CoronNg19}.
\end{remark}
\begin{remark} \rm In the case $m  = 2$,  we also established \cite[Theorem 1.1]{CoronNg19} the following result on the optimality of $T_{opt}$: If  $B \in {\mathcal B}$, $B_{k1} \neq 0$, $\Sigma$ is constant,  and  $(T_{opt} =  \tau_k + \tau_{k+2} = \tau_{k-1} + \tau_{k+1}$ if $k \ge 2$ and $T_{opt} =  \tau_1 + \tau_3 = \tau_2$ if $k = 1)$, then there exists a non-zero constant matrix $C$ such that the system is {\it not} null-controllable  at  the time  $T_{opt}$.
\end{remark}

Concerning the exact controllability, we have

\begin{theorem}\label{thm2}
Assume that $m \ge k \ge 1$,  \eqref{relation-lambda} and \eqref{cond-lambda} hold.  We have
\begin{itemize}
\item[i)] for each $B \in {\mathcal B}_e$, outside a discrete set of $\gamma$ in $\mR$, the control system \eqref{Sys-1} with $C$ replaced by $\gamma C$, \eqref{bdry-0}, and \eqref{bdry-1} is exactly controllable at the time $T_{opt}$.

\item[ii)] for each $\gamma$ outside a discrete set in $\mR$, outside a set of zero measure of $B$ in ${\mathcal B}_e$, the control system \eqref{Sys-1} with $C$ replaced by $\gamma C$, \eqref{bdry-0}, and \eqref{bdry-1} is exactly controllable at the time $T_{opt}$.
\end{itemize}
\end{theorem}

 For the exact controllability and $T>T_{opt}$, the generic assumption was removed in \cite{HO19} (see also \cite[Theorem 4]{CoronNg19-2}), where the following theorem is established
 \begin{theorem}\label{thm4}
Let $m \ge k \ge 1$.   Assume that $B \in \cB_e$ defined in \eqref{def-Be}. The control system \eqref{Sys-1}, \eqref{bdry-0}, and \eqref{bdry-1} is exactly controllable at any time $T$ greater than $T_{opt}$.
\end{theorem}
(In fact \cite{HO19} gives the optimal time of exact controllability even if $B \not \in \cB_e$.)

 For the null controllability and $T>T_{opt}$, we removed the generic assumption was in \cite[Theorems 2]{CoronNg19-2}, where the following theorem is established
\begin{theorem}\label{thm3} Let $k \ge m \ge 1$. Assume  that  $B \in \cB$ defined in \eqref{def-B}. The control system \eqref{Sys-1}, \eqref{bdry-0}, and \eqref{bdry-1} is null-controllable at any time $T$ greater than $T_{opt}$.
\end{theorem}
\begin{remark} \rm
Related controllability results can be also found in \cite{LongHu,Weck}.
\end{remark}

Concerning the optimality of $T_{opt}$, we can prove the following result \cite[Proposition 1.6]{CoronNg19}

\begin{proposition} \label{pro-C} Assume that $C \equiv 0$ and   \eqref{cond-B-1} holds for $1 \le i \le  \min\{k, m\}$,  then, for any $T < T_{opt}$, there exists an initial datum such that $u(T, \cdot) \not \equiv 0$ for every control.
\end{proposition}

\begin{remark} The controls used in \Cref{thm1,thm2,thm3,thm4} are of the form
\begin{equation}\label{controls}
w_+(t, 1) = \sum_{r=1}^R A_r(t) w(t, x_r) + \int_0^1 M(t, y) w(t, y) \, dy + h(t),
\end{equation}
where  $R\in \mN$,  $A_r: [0, T] \to \mR^{m \times n}$, $x_r \in [0, 1]$ $(1 \le r \le R)$, $M: [0, T] \times [0, 1] \to \mR^{n \times n}$, and $h \in [L^\infty(0, T)]^m$. Moreover, the following conditions hold:
\begin{equation}\label{assumption-1}
x_r < c < 1 \mbox{ for some constant $c$},
\end{equation}
\begin{equation}\label{assumption-2}
A_r \in [L^\infty  (0, T)]^{n\times n},  \quad
M \in [L^\infty \big((0, T) \times (0, 1) \big)]^{n\times n},
\end{equation}
for $1 \le r \le R$. The well-posedness of \eqref{Sys-1}, \eqref{bdry-0}, and \eqref{controls} for broad solutions (see \cite[Definition 3.1]{CoronNg19}) was established in \cite[Lemma 3.2]{CoronNg19}.
\end{remark}

We next discuss homogeneous quasilinear hyperbolic systems. More precisely, we consider the equation, for $(t, x)  \in [0, + \infty) \times (0, 1)$,
\begin{equation}\label{Sys-1-QL}
\partial_t w (t, x) =  \Sigma \big(x, w(t, x) \big) \partial_x w (t, x),
\end{equation}
instead of \eqref{Sys-1}, and the boundary and control conditions \eqref{bdry-0} and \eqref{bdry-1}. We assume that  $\Sigma(x, y)$ for $x \in [0, 1]$ and   $y \in \mR^n$ is of the form
\begin{multline}\label{form-A-new}
\Sigma(x, y) = \mbox{diag} \Big(- \lambda_1(x, y), \cdots, - \lambda_{k}(x, y), \\
  \lambda_{k+1}(x, y), \cdots,  \lambda_{k+m}(x, y) \Big),
\end{multline}
where
\begin{multline}\label{relation-lambda-*}
-\lambda_{1}(x, y) < \cdots <  - \lambda_{k} (x, y) \\
< 0 < \lambda_{k+1}(x, y) < \cdots <\lambda_{k+m}(x, y).
\end{multline}
We assume,  for $1 \le i \le n = k + m$,
\begin{equation}\label{cond-lambda-*}
\mbox{$\lambda_i$  is of class $C^2$ with respect to $x$ and $y$}
\end{equation}

Concerning the quasilinear system \eqref{Sys-1-QL}, \eqref{bdry-0}, and \eqref{bdry-1}, we prove \cite[Theorem 1.1]{CoronNg20}:

\begin{theorem}  \label{thm5}  Assume that $B \in {\mathcal B}$.  For any $T > T_{opt}$, there exist $\eps > 0$ and a time-independent feedback control for \eqref{Sys-1-QL}, \eqref{bdry-0}, and \eqref{bdry-1}  such that if the compatibility conditions $($at $x =0$$)$  \eqref{compatibility-0} and \eqref{compatibility-1} below hold for $w(0, \cdot)$,   and
$$
\| w(0, \cdot) \|_{C^1([0, 1])} < \eps,
$$
then   $w(T, \cdot) =0$.
\end{theorem}

The compatibility conditions considered in \Cref{thm1} are:  \begin{equation}\label{compatibility-0}
w_-(0, 0) = B  \big( w_+(0, 0) \big)
\end{equation}
and
\begin{multline}\label{compatibility-1}
\Sigma_- \big(0, w(0, 0)\big) \partial_x w_-(0, 0) \\
= B \big( w_+(0, 0) \big)  \Sigma_+ \big(0, w(0, 0)\big) \partial_x w_+(0, 0).
\end{multline}
Here we denote
$$
\mbox{$w_- = (w_1, \cdots, w_k)\tr, \quad w_+ = (w_{k+1}, \cdots, w_{k+m})\tr$.}
$$
$$
\Sigma_-(x, y)  = \mbox{diag} \Big(- \lambda_1(x, y), \cdots, - \lambda_{k}(x, y) \Big)$$
and
$$
\Sigma_+(x, y)  = \mbox{diag} \Big( \lambda_{k+1}(x, y), \cdots,  \lambda_{n}(x, y) \Big).
$$

\begin{remark} \label{rem-NL} \rm In \cite{CoronNg20}, we also consider nonlinear boundary condition at $x =0$, i.e., instead of \eqref{bdry-0}, we deal with
\begin{equation*}
w_-(t, 0)  = \B \big(w_+(t,  0) \big) \mbox{ for } t \ge 0,
\end{equation*}
for some
$$
\mbox{$\B \in \big( C^2(\mR^m) \big)^k$ \mbox{with} $\B(0) = 0$,}
$$
\Cref{thm5} also holds if the condition $B \in \cB$ is replaced by $\nabla B(0) \in \cB$.

\end{remark}

\begin{remark}\rm The feedbacks constructed in \cite{CoronNg20} use additional $4m$ state-variables (dynamics extensions) to avoid imposing compatibility conditions at $x=1$.
\end{remark}

In our recent work \cite{CoronNg20-L}, we present Lyapunov's  functions for the feedbacks given in \Cref{thm5} and use estimates for Lyapunov's functions to rediscover the finite stabilization result.

\section{The ideas of the proof of \Cref{thm1}-\Cref{thm3}}

The starting point  of our analysis is the backstepping approach.  The key idea of the backstepping approach is to make the following change of variables
\begin{equation}\label{backstepping}
u(t, x) = w(t, x) - \int_0^x K(x,y) w(t, y) \, dy,
\end{equation}
for some kernel $K: \D \to \mR^{n \times  n}$  which is chosen in such a way that the system for $u$ is easier to control. Here
\begin{equation}\label{def-D}
\D = \big\{(x, y) \in (0, 1)^2; 0  <  y <  x \big\}.
\end{equation}
To determine/derive the equations for  $K$, we first compute $\partial_t u (t, x) - \Sigma(x) \partial_x u(t, x)$. Taking into account \eqref{backstepping},
we formally have \footnote{We assume here that $u$, $w$, and $K$ are smooth enough so that the below computations make sense.}
\begin{multline*}
\partial_t u (t, x)
=   \partial_t w(t, x) - K(x, x) \Sigma(x) w(t, x) + K(x, 0) \Sigma(0) w(t, 0) \\
+ \int_0^x \Big[ \partial_y  \big(K(x,y)  \Sigma(y) \big) w(t, y) - K(x, y) C(y) w(t, y) \Big] \, dy,
\end{multline*}
and
\begin{align*}
\partial_x u(t, x) = \partial_x w(t, x) - \int_0^x \partial_x K(x,y) w(t, y) \, dy - K(x,x) w(t, x).
\end{align*}
It follows from \eqref{Sys-1} that
\begin{multline}
\partial_t u (t, x) - \Sigma(x) \partial_x u(t, x) \\
=  \Big( C(x) - K(x,x) \Sigma(x)  + \Sigma(x) K(x, x) \Big) w(t, x) \\ + K(x, 0) \Sigma(0) u(t, 0)
+ \int_0^x \Big[ \partial_y K(x, y) \Sigma (y) + K(x, y)  \Sigma'(y)  \\
- K(x, y) C(y) + \Sigma(x) \partial_x K(x, y) \Big] w(t, y) \, dy.
\end{multline}

We search a  kernel $K$ which satisfies  the following two conditions
\begin{multline}\label{equation-K}
\partial_y K(x, y) \Sigma (y) +  \Sigma(x) \partial_x K(x, y) \\
+  K(x, y)  \Sigma'(y) - K(x, y) C(y)= 0 \mbox{ in } \D
\end{multline}
and, for $x \in (0, 1)$,
\begin{equation}\label{bdry1-K}
{\mathcal C}(x): = C(x) - K(x,x) \Sigma(x)  + \Sigma(x) K(x, x) = 0,
\end{equation}
so that one formally has
\begin{multline}\label{eq-u-*}
\partial_t u (t, x) =  \Sigma(x) \partial_x u (t, x) \\
+ K(x, 0) \Sigma(0)  u(t, 0)  \mbox{ for } (t, x)  \in \mR_+ \times (0, 1).
\end{multline}
Set
\begin{equation}\label{def-Q}
Q := \left(\begin{array}{cccc} 0_k & B \\[6pt]
0_{m, k} & I_m
\end{array}\right).
\end{equation}
and
\begin{equation}\label{def-S-*}
S(x): = K(x, 0) \Sigma (0) Q.
\end{equation}
Here and in what follows, $0_{i, j}$  denotes the zero matrix  of size $i \times j$ for $i, \, j \in \mN$, and $M_{pq}$ denotes the $(p, q)$-component of a matrix $M$.
From \eqref{def-Q}, the matrix $S \in [L^\infty(0, 1)]^{n \times n}$ has the structure
\begin{equation}\label{cond-S}
S = \left( \begin{array}{cl}
0_{k, k} & S_{-+} \\[6pt]
0_{m, k} & S_{++}
\end{array}\right),
\end{equation}

Using \eqref{bdry-0}, equation \eqref{eq-u-*} becomes, for  $(t, x)  \in \mR_+ \times (0, 1)$,
\begin{equation}\label{eq-u}
\partial_t u (t, x) =  \Sigma(x) \partial_x u (t, x) + S(x)  u(t, 0).
\end{equation}
We are able to  show that such a $K$ exists so that \eqref{eq-u} holds \cite[Lemma 3.3]{CoronNg19}; moreover, $K$ can be chosen in such a way that
\begin{equation}\label{cond-S-new}
\quad (S_{++})_{pq}(x) = 0 \mbox{ for } 1 \le q \le p \le m,
\end{equation}
this point turns out to be important for our analysis.  It is shown in \cite[Proposition 3.1]{CoronNg19} that  the null-controllability and the exact controllability   of \eqref{Sys-1}, \eqref{bdry-0}, and \eqref{bdry-1} at the time $T$ can be derived from the null-controllability and the exact controllability at the time $T$ of \eqref{eq-u} equipped  the boundary condition at $x=0$
\begin{equation}\label{bdry-u}
u_- (t, 0)  = B u_+(t, 0) \mbox{ for } t \ge 0,
\end{equation}
and the boundary controls  at $x = 1$
\begin{equation}\label{control-u}
u_+ = U(t) \mbox{ for } t \ge 0 \mbox{ where $U$ is the control.}
\end{equation}

In what follows, we discuss the idea of the proof of \Cref{thm1} and \Cref{thm3}. The proofs of \Cref{thm2} and \Cref{thm4} (given in \cite[Theorem 4]{CoronNg19-2}) are in the same spirit of these ones and not dealt with.

%To establish the null-controllability for $u$, we use the Hilbert uniqueness method which involves crucially a compactness result type in  \Cref{lem-compact} with its roots in \cite{CoronNg19}.

\subsection{On the proof of \Cref{thm1}} The idea of the proof is to derive the sufficient conditions for which the null-controllability holds. Using the characteristic method, these conditions will be written under the form $U + {\mathcal K} U = {\mathcal F}$ where ${\mathcal K}$ is an analytic compact operator with respect to $\lambda$ where $C$ is replaced by $\lambda C$. We then apply the Fredholm theory to obtain the conclusion. The process to derive the equation $U + {\mathcal K} U = {\mathcal F}$ is somehow involved. We refer the reader to \cite{CoronNg19} for the details.

\subsection{On the proof of \Cref{thm3}} As mentioned above the null-controllability of \eqref{Sys-1}, \eqref{bdry-0}, and \eqref{bdry-1} at the time $T$ is equivalent to the one \eqref{eq-u}-\eqref{control-u}  at the same time. The proof of the null-controllability of the later system  is based on the Hilbert uniqueness method given in the following result \cite[Lemma 1]{CoronNg19-2} whose proof is standard.

\begin{lemma}\label{lem-observability} Let $T>0$. System \eqref{eq-u}-\eqref{control-u} is null controllable at the time $T$ if and only if, for some positive constant $C$,
\begin{equation}\label{observability}
\int_{-T}^0 |v_+(t, 1)|^2 \, dt \ge C \int_0^1 |v(-T, x)|^2 \, dx \; \; \; \forall \, v \in [L^2(0, 1)]^n,
\end{equation}
where $v(\cdot, \cdot)$ is the unique solution of the system, for $(t, x)  \in (-\infty, 0) \times (0, 1)$,
\begin{equation}\label{eq-v-O}
\partial_t v (t, x) =  \Sigma(x) \partial_x v (t, x) + \Sigma'(x)  v(t, x),
\end{equation}
with, $t < 0$,
\begin{equation}\label{bdry-v0-O}
v_-(t, 1)  = 0,
\end{equation}
\begin{multline}\label{bdry-v-O}
\Sigma_+ (0) v_+(t, 0) = - B\tr  \Sigma_- (0) v_- (t, 0) \\
+ \int_0^1  S_{-+}\tr(x) v_- (t, x) + S_{++}\tr(x) v_+(t, x) \, dx,
\end{multline}
and
\begin{equation}\label{initial-v-O}
v(t = 0, \cdot) = v \mbox{ in } (0, 1).
\end{equation}
\end{lemma}

The next key part of the proof of \Cref{thm3} is the following compactness result for \eqref{eq-v-O}-\eqref{bdry-v-O} \cite[Lemma 4]{CoronNg19-2}; the structure of $S$ given in \eqref{cond-S-new} plays a role in the proof.

\begin{lemma}\label{lem-compact} Let $k \ge m \ge 1$, $B \in \cB$,  and $T \ge T_{opt}$. Assume that
$(v_N)$ be a sequence of solutions of \eqref{eq-v-O}-\eqref{bdry-v-O} $($with $v_N(0, \cdot)$ in $[L^2(0, 1)]^n$$)$ such that
\begin{gather}\label{thm2-generating-eq-a}
\sup_{N} \|v_N (-T, \cdot) \|_{L^2(0, 1)} < + \infty,
\\
\label{thm2-generating-eq-b}
\lim_{N \to + \infty} \|v_{N, +}(\cdot, 1) \|_{L^2(-T, 0)} = 0.
\end{gather}
We have, up to a subsequence,
\begin{equation}\label{thm2-claim1}
v_N(-T, \cdot) \mbox{ converges in } L^2(0, 1),
\end{equation}
and  the limit $V \in [L^2(0, 1)]^n$ satisfies the equation
\begin{equation}\label{thm2-claim2}
V =  {\mathcal K} V,
\end{equation}
for some compact operator $\mathcal{K}$ from  $[L^2(0, 1)]^n$ into itself. Moreover, $\cK$ depends only  on $\Sigma$, $S$, and $B$; in particular, $\cK$ is independent of $T$.
\end{lemma}

\begin{proof}[Proof of \Cref{thm3}] For $T> T_{opt}$,  set
\begin{multline}\label{thm2-def-YT}
Y_T := \Big\{ V \in L^2(0, 1): V \mbox{ is the limit in $L^2(0, 1)$ of } \\[6pt]
\mbox{some subsequence of solutions $\big(v_N(-T, \cdot) \big)$ }\\[6pt]  \mbox{of  \eqref{eq-v-O}-\eqref{bdry-v-O} such that
 \eqref{thm2-generating-eq-a} and \eqref{thm2-generating-eq-b} hold}\Big\}.
\end{multline}
It is clear that $Y_T$ is a vectorial space. Moreover, by \eqref{thm2-claim2} and the compact property of $\cK$, we have
\begin{equation}\label{thm2-pro-YT-1}
\dim Y_T \le  C,
\end{equation}
for some positive constant $C$ independent of $T$.

We next show that
\begin{equation}\label{thm2-pro-YT-2}
Y_{T_2} \subset Y_{T_1} \mbox{ for }  T_{opt} < T_1 < T_2.
\end{equation}
Indeed,  let $V \in Y_{T_2}$. There exists a sequence of solutions $(v_N)$  of \eqref{eq-v-O}-\eqref{bdry-v-O} such that
\begin{equation}
\left\{\begin{array}{c}
v_N (-T, \cdot) \to V  \mbox{ in } L^2(0, 1),  \\[6pt]
\lim_{N \to + \infty} \|v_{N, +}(\cdot, 1) \|_{L^2(-T_2, 0)} = 0.
\end{array} \right.
\end{equation}
By considering the sequence $v_N(\cdot - \tau, \cdot)$ with $\tau = T_2 - T_1$, we derive that $V \in Y_{T_1}$.

The arguments are then in the spirit of \cite{BLR92} (see also \cite{Rosier97})  via an eigenvalue problem in finite dimension using a contradiction argument.   By \Cref{lem-observability}, to obtain the null-controllability at the time $T > T_{opt}$, it suffices to  prove \eqref{observability} by contradiction. Assume that  there exists a sequence of solutions $(v_N)$ of  \eqref{eq-v-O}-\eqref{bdry-v-O} such that
\begin{equation}\label{observability-contradiction}
N \int_{-T}^0 |v_{N, +}(t, 1)|^2 \, dt \le \int_0^1 |v_N(-T, x)|^2 \, dx =1.
\end{equation}
By \eqref{thm2-claim1}, up to a subsequence, $v_N(-T, \cdot)$ converges in $L^2(0, 1)$ to a limit $V$. It is clear that $\|V \|_{L^2(0, 1)} = 1$; in particular, $V \neq 0$. Consequently,
\begin{equation}\label{thm2-pro-YT-3}
Y_T \neq \{ 0\}.
\end{equation}

By \eqref{thm2-pro-YT-1}, \eqref{thm2-pro-YT-2},  and \eqref{thm2-pro-YT-3}, there exist $T_{opt} < T_{1} < T_2 < T$ such that
\begin{equation*}
\dim Y_{T_1} = \dim Y_{T_2} \neq 0.
\end{equation*}

We can prove  that, for $V \in Y_{T_1}$,
\begin{equation}\label{thm2-claimB}
 \Sigma \partial_x V + \Sigma' V  \mbox{ is an element in } Y_{T_1}.
\end{equation}
Recall that $Y_{T_1}$ is real and  of finite dimension. Consider its natural extension as a complex vectorial space and still denote its extension by $Y_{T_1}$. Define
\begin{equation*}
\begin{array}{cccc}
\cA: & Y_{T_1} & \to  & Y_{T_1} \\[6pt]
& V & \mapsto & \Sigma \partial_x V + \Sigma' V.
\end{array}
\end{equation*}
From the definition of $Y_{T_1}$, it is clear that, for $V \in Y_{T_1}$,
\begin{equation}\label{thm2-bdry-V0}
V_-(1)  = 0
\end{equation}
and
\begin{multline}\label{thm2-bdry-V}
\Sigma_+ (0) V_+(0) = - B\tr  \Sigma_- (0) V_- (0) \\
 + \int_0^1  S_{-+}\tr(x) V_- (x) + S_{++}\tr(x) V_+(x) \, dx.
\end{multline}
Since $Y_{T_1} \neq \{0 \}$ and $Y_{T_1}$ is of finite dimension, there exists $\lambda \in \mC$ and $V \in Y_{T_1} \setminus \{0 \}$ such that
\begin{equation*}
\cA V = \lambda V.
\end{equation*}
Set
$$
v(t, x) = e^{\lambda t} V(x) \mbox{ in } (-\infty, 0) \times (0, 1).
$$
Using \eqref{thm2-bdry-V0} and \eqref{thm2-bdry-V}, one can verify that $v(t, x)$ satisfies \eqref{eq-v-O}-\eqref{bdry-v-O}.
Applying the characteristic method, one can deduce that  $v(t, \cdot ) =  0 $ in $(0, 1)$ for $t < - \tau_{k+1} - \cdots - \tau_{k+m} $. It follows that $V = 0$ which contradicts the fact $V \neq 0$.
Thus \eqref{observability} holds and the null-controllability  is valid for $T > T_{opt}$. The details can be found in \cite{CoronNg19-2}.
\end{proof}

\section{On the proof of \Cref{thm5}}

We will also deal with the nonlinear boundary condition at $x = 0$ as mentioned in \Cref{rem-NL}. We first consider the case $m > k$.
 Consider the last equation of \eqref{bdry-0} and  impose the condition $w_{k}(t, 0) = 0$. Using \eqref{cond-B-1} with $i = 1$ and the implicit function theorem, one can then write the last  equation of \eqref{bdry-0} under the form
\begin{equation}\label{kd-1}
w_{m+k}(t, 0) = M_k  \Big( w_{k +1 }(t, 0), \cdots, w_{m + k - 1} (t, 0) \Big),
\end{equation}
for some $C^2$ nonlinear map  $M_k$ from $U_k$ into $\mR$ for some neighborhood $U_k$ of $0 \in \mR^{m-1}$ with $M_k(0) = 0$ provided that $|w_{+}(t, 0)|$ is sufficiently small.

Consider the last two equations of \eqref{bdry-0} and impose the condition $w_{k}(t, 0) = w_{k-1}(t, 0) = 0$. Using \eqref{cond-B-1} with $i = 2$ and  the Gaussian elimination approach,  one can then write these two equations under the form \eqref{kd-1} and
\begin{equation}\label{kd-2}
w_{m+k-1}(t, 0) = M_{k-1} \Big( w_{k +1 }(t, 0), \cdots, w_{m + k - 2} (t, 0) \Big),
\end{equation}
for some $C^2$ nonlinear map  $M_{k-1}$ from $U_{k-1}$ into $\mR$ for some neighborhood $U_{k-1}$ of $0 \in \mR^{m-2}$ with $M_{k-1}(0) = 0$ provided that $|w_{+}(t, 0)|$ is sufficiently small, etc. Finally, consider the  $k$ equations of \eqref{bdry-0} and impose the condition $w_{k}(t, 0) = \dots = w_{1}(t, 0) = 0$. Using \eqref{cond-B-1} with $i = k$ and the Gaussian elimination approach,  one can then write these $k$ equations under the form \eqref{kd-1}, \eqref{kd-2}, \dots, and \begin{equation}\label{kd-k}
w_{m+1}(t, 0) = M_{1} \Big(w_{k +1 }(t, 0), \cdots, w_{m} (t, 0) \Big),
\end{equation}
for some $C^2$ nonlinear map  $M_1$ from $U_1$ into $\mR$ for some neighborhood $U_1$ of $0 \in \mR^{m-k}$ with $M_1(0) = 0$ provided that $|w_{+}(t, 0)|$ is sufficiently small. These nonlinear maps  $M_1, \dots, M_k$ will be used in the construction of feedbacks.

Define
$$
\frac{d}{dt}x_j(t, s, \xi) =   \lambda_j \Big(x_j(t, s, \xi), w \big(t, x_j(t, s, \xi) \big) \Big)
$$
with  $x_j(s, s, \xi) = \xi$ for  $1 \le j \le k$, and
$$
\frac{d}{dt}x_j(t, s, \xi) = -  \lambda_j \Big(x_j(t, s, \xi), w \big(t, x_j(t, s, \xi) \big) \Big),
$$
with  $x_j(s, s, \xi) = \xi$ for  $k+1 \le j \le k + m$. We do not precise at this stage the domain of the definition of $x_j$. Later, we only consider the flows in the regions where the solution $w$ is well-defined.

To arrange the compatibility of our controls, we also  introduce following auxiliary variables
satisfying autonomous dynamics.  Set $\delta = T- T_{opt} > 0$. For $t \ge 0$, define, for $k+1 \le j \le k+m$,
\begin{equation}
\zeta_{j}(0) = w_{0, j}(0), \;  \zeta_j'(0) = \lambda_j\big(0, w_0(0)\big) w_{0, j}'(0),
\end{equation}
and
\begin{equation}
\eta_j(0) = 1, \quad \eta_j'(0)=0,
\end{equation}
$$
 \zeta_{j}(t) = \eta_{j}(t) = 0 \mbox{ for } t \ge \delta/2,
$$
The feedback is then chosen as follows:
\begin{multline}\label{bdry-1-NL}
w_{m+ k}(t, 1) =  \zeta_{m+k}(t)  \\[6pt]+ ( 1 - \eta_{m+k}(t)) M_{k} \Big(w_{k+ 1}\big(t, x_{k+1} (t, t+ t_{m+k},  0) \big), \dots, \\w_{k+ m-1}\big(t, x_{k+ m - 1} (t, t+  t_{m + k }, 0) \big)\Big)
\end{multline}
\begin{multline}\label{bdry-2-NL}
w_{m+ k - 1}(t, 1) = \zeta_{m+k-1}(t)  \\[6pt]
+  ( 1- \eta_{m+k-1}(t)) M_{k-1} \Big(w_{k+ 1}\big(t, x_{k+1} (t, t+ t_{m + k - 1}, 0) \big), \dots, \\
w_{k+ m-2}\big(t, x_{k+ m - 2} (t, t+  t_{m + k -1}, 0) \big) \Big)
\end{multline}
\dots
\begin{multline}\label{bdry-m-NL}
w_{m+ 1}(t, 1) =  \zeta_{m+1}(t) \\[6pt]
+ (1 -  \eta_{m+1}(t))  M_{1} \Big(w_{k+ 1} \big(t, x_{k+1} (t, t + t_{m + 1}, 0) \big), \dots, \\
w_{m}\big(t, x_{m} (t, t+  t_{m +1}, 0) \big)\Big)
\end{multline}
and,  for $k+1 \le j \le m$,
\begin{equation}\label{bdry-k-m-NL}
w_{j}(t, 1) = \zeta_j(t),
\end{equation}

This feedback is well-determined by noting that \eqref{bdry-1-NL} depends only on the current state, \eqref{bdry-2-NL} depends only on the current state and \eqref{bdry-1-NL}, etc.

For $m \le k$,  the  feedback law is given as follows:
\begin{multline*}
w_{m+ k}(t, 1) =  \zeta_{m+k}(t)  \\
+  (1 - \eta_{m+k}(t)) M_{k} \Big(w_{k+ 1}\big(t, x_{k+1} (t, t+ t_{m+k},  0) \big), \dots, \\
w_{k+ m-1}\big(t, x_{k+ m - 1} (t, t+  t_{m + k }, 0) \big)\Big),
\end{multline*}
\dots
\begin{multline*}
w_{k+2}(t, 1) = \zeta_{k+2}(t)  \\
+  ( 1-\eta_{k+2}(t))  M_{2} \Big(w_{k+ 1}\big(t, x_{k+1} (t, t+ t_{k+2}, 0) \big) \Big),
\end{multline*}
and
$$
w_{k+1}(t, 1) = \zeta_{k+1}(t).
$$

The null-controllability for small initial data can be derived from the properties of $M_j$ for $1 \le j \le k$. An key technical part of the proof is the well-posedness for \eqref{Sys-1-QL} and  \eqref{bdry-0}  equipped these feedback laws, see \cite[Lemma 2.2]{CoronNg20}.

\medskip
\noindent{\bf Conclusion:} This paper is devoted to the null-controllability, exact
controllability, and stabilization of hyperbolic systems for the optimal time. The
starting point of the analysis in the inhomogeneous case is based on the backstepping approach.
The ideas of the analysis are  presented.

\providecommand{\bysame}{\leavevmode\hbox to3em{\hrulefill}\thinspace}
\providecommand{\MR}{\relax\ifhmode\unskip\space\fi MR }
\providecommand{\MRhref}[2]{%
  \href{http://www.ams.org/mathscinet-getitem?mr=#1}{#2}
}
\providecommand{\href}[2]{#2}

\end{document}